\documentclass[11pt]{amsart}
\usepackage{amsfonts,amsmath,latexsym,amssymb,euscript,graphicx,units,mathrsfs,
amsthm}

\usepackage{caption,subcaption}

\usepackage{hyperref}

\usepackage{graphicx}
\usepackage{color}
\usepackage{amssymb}
\usepackage[T1]{fontenc}
\usepackage{latexsym}
\usepackage{xypic}
\usepackage{eufrak}
\usepackage{euscript}
\usepackage{amsfonts,amsmath}
\usepackage{verbatim}
\usepackage{fancyhdr}
\usepackage[english]{babel}
\usepackage{mathrsfs}
\usepackage{units}

\newtheorem{prop}{Proposition}[section]
\newtheorem{cor}[prop]{Corollary}

\newtheorem{remark}[prop]{Remark}

\newtheorem{thm}[prop]{Theorem}

\renewcommand{\geq}{\geqslant}
\def\leq{\leqslant}

\newcommand{\N}{\mathbb{N}}

\newcommand{\R}{\mathbb{R}}

\newcommand{\E}{\mathbb{E}}

\newcommand{\eq}{\begin{equation}}
\newcommand{\qe}{\end{equation}}

\begin{document}

\title{Regularity of solutions of the Stein equation and rates in the multivariate central limit theorem} 

%\author{T.O. Gallou\"et, G. Mijoule and Y. Swan}

\author{T. O. Gallouet}\address{Thomas O. Gallou\"et: Inria, Project team  MOKAPLAN and Mathematics department, Universit\'e de Li\`ege, Belgium, (\href{mailto:thomas.gallouet@inria.fr}{\tt thomas.gallouet@inria.fr})}   
\author{G. Mijoule}\address{Guillaume Mijoule: Inria, Project team  MOKAPLAN and Mathematics department, Universit\'e de Li\`ege, Belgium, (\href{mailto: guillaume.mijoule@gmail.com}{\tt guillaume.mijoule@inria.fr})}   
\author{Y. Swan}\address{Yvik Swan: Mathematics department, Universit\'e de Li\`ege, Belgium, (\href{mailto: yswan@ulg.ac.be}{\tt yswan@ulg.ac.be})}   

\begin{abstract}
Consider the multivariate Stein equation
$\Delta f - x\cdot \nabla f = h(x) - \E h(Z)$, where $Z$ is a
standard $d$-dimensional Gaussian random vector, and let $f_h$ be
the solution given by Barbour's generator approach. We prove that,
when $h$ is $\alpha$-H\"older ($0<\alpha\leq1$), all derivatives of
order $2$ of $f_h$ are $\alpha$-H\"older {\it up to a $\log$ factor}; in particular they are $\beta$-H\"older for all
$\beta \in (0, \alpha)$, hereby improving existing regularity
results on the solution of the multivariate Gaussian Stein
equation. For $\alpha=1$, the regularity we obtain is optimal, as shown by an example given by Rai\v{c} \cite{raivc2004multivariate}. As an application, we prove a near-optimal Berry-Esseen
bound of the order $\log n/\sqrt n$ in the classical multivariate
CLT in $1$-Wasserstein distance, as long as the underlying random
variables have finite moment of order $3$. When only a finite moment
of order $2+\delta$ is assumed ($0<\delta<1$), we obtain the
optimal rate in $\mathcal O(n^{-\frac{\delta}{2}})$. All constants
are explicit and their dependence on the dimension $d$ is studied
when $d$ is large.
\end{abstract}

\maketitle

\noindent
{\small {\bf Keywords.}
Berry-esseen bounds; Stein's method; Elliptic regularity; 
\vspace{5pt}

\noindent
{\bf AMS subjects classification. }
}

\section{Introduction}

\subsection{Multivariate Stein's method}

Stein's method is a powerful tool for estimating distances between
probability distributions. It first appeared in \cite{stein1972bound},
where the method was introduced for the purpose of comparison with a
(univariate) Gaussian target. The idea, which still provides the
backbone for the contemporary instantiations of the method, is as
follows. If $Z$ is a standard Gaussian random variable, then 
\begin{equation}
\label{eq:1}
\E[f'(Z)-Zf(Z)]=0
\end{equation}
for all absolutely continuous functions $f$ with
$\E|f'(Z)|<+\infty$. Let $X$ be another random variable, and consider
the integral probability distance between the laws of $X$ and $Z$
given by
\begin{equation}
\label{eq:defdistance}
d_{\mathcal H}(X,Z) = \sup_{h\in \mathcal H} \E[h(X)-h(Z)],
\end{equation}
for $\mathcal H$ a class of tests functions which are integrable with
respect to the laws of both $X$ and $Z$. Many classical distances admit a
representation of the form \eqref{eq:defdistance}, including the
Kolmogorov (with $\mathcal{H}$ the characteristic functions of
half-lines), total variation (with $\mathcal{H}$ the characteristic
functions of Borel sets), and $1$-Wasserstein a.k.a.\ Kantorovitch
(with $\mathcal{H}$ the 1-Lipschitz real functions) distances; see
e.g.\ \cite{muller1997integral}. Letting $\omega$ denote the standard Gaussian
pdf, we define for every $h \in \mathcal{H}$ the function
\begin{equation}\label{eq:5}
f_h(x) = \frac{1}{\omega(x)}\int_{-\infty}^x \left( h(y) - \E h(Z)
\right) \omega(y) dy.
\end{equation}
This function is a solution to the ODE (called a \emph{Stein
equation})
\begin{equation}
\label{eq:steinclassic}
f_h'(x)-xf_h(x) = h(x) - \E h(Z), \quad x \in \R,
\end{equation}
which allows to rewrite integral probability metrics
\eqref{eq:defdistance} as
$$d_{\mathcal{H}}(X,Z) = \sup_{h\in \mathcal H}
\E[f_h'(X)-Xf_h(X)]. $$ 
Stein's intuition was to exploit this
last identity to estimate the distance between the laws of $X$ and
$Z$. One of the reasons for which the method works is the fact that
the function $f_h$ defined in \eqref{eq:5} enjoys many regularity
properties. For instance, one can show (see e.g.\ \cite[pp.\
65-67]{nourdin2012normal}) that if $h$ is absolutely continuous then
\begin{equation}
\label{eq:4}
\|f_h\|_{\infty}\le 2\|h'\|_{\infty}, \quad \|f_h'\|_{\infty}\le
\sqrt{2/\pi}\|h'\|_{\infty} \mbox{ and } \|f_h''\|_{\infty}\le 2\|h'\|_{\infty},
\end{equation}
$\| \, . \, \|_\infty$ holding for the supremum norm. This offers a
wide variety of handles on $d_{\mathcal{H}}(X, Z)$ -- typically via
low order Taylor expansion arguments -- for all important choices of
test functions $\mathcal{H}$ and under weak assumptions on $X$. This
observation has been used, for instance, to obtain Berry-Esseen-type
bounds in the classical central limit theorem in 1-Wasserstein
distance, Kolmogorov or total variation distances, see
\cite{chen2010normal,nourdin2012normal}.

Consider now a $d$-dimensional Gaussian target
$Z \sim \mathcal N(0,I_d)$. The $d$-dimensional equivalent to
identity \eqref{eq:1} was identified in
\cite{barbour1990stein,gotze1991rate} as 
$$\E [ \Delta f(Z) - Z \cdot \nabla f(Z)] = 0,$$
which holds for a ``large class'' of functions
$f \; : \, \R^d \rightarrow \R$ ($x\cdot y$ denotes the usual scalar product
between vectors $x,y\in\R^d$). We will define the ``large class'' of
functions precisely in Proposition~\ref{prop:regul-solut-steins}
below. For $h$ a function with finite Gaussian mean, the multivariate
Stein equation then reads
\begin{equation}
\label{eq:steinmulti}
\Delta f(x) - x \cdot \nabla f(x) = h(x) - \E h(Z), \quad x \in \R^d.
\end{equation}
Note that \eqref{eq:steinmulti} is a second order equation in the
unknown function $f$; in dimension $d=1$, \eqref{eq:steinmulti}
reduces to $f''(x) - x f'(x) = h(x) - \E h(Z)$ which is obtained by
applying \eqref{eq:steinclassic} to $f'$. Barbour
\cite{barbour1990stein} identified a solution 
of \eqref{eq:steinmulti} to be
\begin{equation}
\label{eq:barboursolution}
f_h(x) = -\int_0^1 \frac{1}{2t} \E [h (\sqrt t x + \sqrt{1-t} Z)
-h(Z)]\;dt,
\end{equation} 
and the same argument as in the 1-dimensional setting leads to the
identity
\begin{equation}\label{eq:6}
d_{\mathcal{H}}(X, Z) = \sup_{h \in \mathcal{H}} \mathbb{E} \left[\Delta f_h(X)
- X \cdot \nabla 
f_h(X) \right],
\end{equation}
which is the starting point for multivariate Gaussian approximation
via Stein's method. The explicit representation
\eqref{eq:barboursolution} is suitable to obtain regularity properties
of $f_h$ in terms of those of $h$; for instance (see e.g.\ \cite[Lemma
2.6]{reinert2009multivariate}) it is known that if $h$ is $n$ times
differentiable then $f_h$ is $n$ times differentiable and
\begin{equation}\label{eq:3}
\left| \frac{\partial^k f_h(x)}{\prod_{j=1}^k \partial x_{i_j}} \right|
\le \frac{1}{k} \left| \frac{\partial^kh(x)}{\prod_{j=1}^k \partial x_{i_j}} \right|, 
\end{equation}
for every $x \in \R^d$. Hence, contrarily to the univariate case where
first order assumption on $h$ was sufficient to deduce second order
regularity for $f_h$ (recall \eqref{eq:4}), a bound such as
\eqref{eq:3} only shows the same regularity for $h$ and $f_h$. In most
practical implementations of the method, however, Taylor
expansion-type arguments are used to obtain the convergence rates from
the rhs of \eqref{eq:6}; hence regularity of $f_h$ is necessary in
order for the argument to work. This restricts the choice of class
$\mathcal{H}$ in which the statements are made and therefore weakens
the strength -- be it only in terms of the choice of distance -- of
the resulting statements.

An important improvement in this regard is due to Chatterjee and
Meckes \cite{chatterjee2008multivariate} who obtained (among other
regularity results) that
\begin{equation}
\label{eq:boundchatterjee}
\underset{x\in\R^d}{\sup} \left\|\nabla^2 f_h(x)\right\|_{H.S} \leq
\|\nabla h\|_\infty, 
\end{equation}
($\| M \|_{H.S.}$ stands for the Hilbert-Schmidt norm of a matrix $M$
and $\nabla^2 f_h$ for the Hessian of $f_h$); see also
\cite{reinert2009multivariate}. Gaunt \cite{gaunt2016rates} later
showed a generalization of this result, namely a version of
\eqref{eq:3} where the derivatives of order $k$ of $f_h$ can be
bounded by derivatives of order $k-1$ of $h$. This still does not
concur with the univariate case as we know that, in this case and when
$h'$ is bounded, one can bound one higher derivative of $f_h$: indeed,
it holds $|f_h^{(3)}| \leq 2 |h'|$ (here the function $f_h$ is the
solution \eqref{eq:barboursolution} to the univariate version of the
second order equation \eqref{eq:steinmulti}). This loss of regularity
is, however, not an artefact of the method of proof but is inherent to
the method itself: Rai\v{c} \cite{raivc2004multivariate} exhibits a
counterexample, namely a Lipschitz-continuous function such that the
second derivative of $f_h$ is \emph{not} Lipschitz-continuous. We will
discuss this example in detail later on.

\subsection{Multivariate Berry-Esseen bounds}

Let $(X_i)_{i\geq 1}$ be an i.i.d.\ sequence of random vectors in
$\R^d$, and for simplicity take them centered with identity covariance
matrix. Let $W = n^{-1/2}\sum_{i=1}^n X_i$, $Z \sim \mathcal N(0,I_d)$
and consider the problem of estimating $ D(W, Z)$, some probability
distance between the law of $Z$ and that of $W$. According to
\cite{gotze1991rate}, the earliest results on this problem in
dimension $d\ge2$ concern distances of the form \eqref{eq:defdistance}
with $\mathcal{H}$ indicator functions of measurable convex sets in
$\R^d$ (which is a multivariate generalization of the Kolmogorov
distance). The best result for this choice of distance is due to
\cite{bentkus2003dependence} where an estimate of the form
$d_{\mathcal{H}}(W, Z) \le 400\, d^{1/4} n^{-1/2} E[|X_1|^3]$ is shown
($|\cdot|$ is the Euclidean norm); the
dependence on the dimension is explicit and the best available for
these moment assumptions and this distance.  More recently, a high
dimensional version of the same problem was studied in
\cite{chernozhukov2017central}, with $\mathcal{H}$ the class of
indicators of hyper-rectangles in $\R^d$; we also refer to the latter
paper for an extensive and up-to-date literature review on such
results.

Another important natural family of probability distances are the
Wasserstein distances of order $p$ (a.k.a.\ Mallows distances) defined
as
\begin{equation}
\label{eq:2}
\mathcal{W}_p(W, Z) = \left(\inf \mathbb{E} \left[ |X_1 - Y_1|^p \right]\right)^{1/p}
\end{equation}
where the infimum is taken over all joint distributions of the random
vectors $X_1$ and $Y_1$ with respective marginals the laws of $W$ and
$Z$. Except in the case $p=1$, such distances cannot be written under
the form \eqref{eq:defdistance}; as previously mentionned, when $p=1$
the distance $\mathcal{W}:=\mathcal{W}_1$ in \eqref{eq:2} is of the form
\eqref{eq:6} with $\mathcal{H}$ the class of Lipschitz function with
constant 1.  Because
$\mathcal{W}_p(\cdot, \cdot) \ge \mathcal{W}_{p'}(\cdot, \cdot)$ for
all $p \ge p'$, bounds in $p$-Wasserstein distance are stronger than
those in $p'$-Wasserstein distance; in particular
$\mathcal{W}_p(\cdot, \cdot) \ge \mathcal{W}_1(\cdot, \cdot)$ for all
$p \ge1$. We refer to \cite{villani2003topics} for more information on
$p$-Wasserstein distances.  CLT's in Wasserstein distance have been
studied, particularly in dimension 1, where we refer to the works
\cite{bobkov2013entropic,rio2009upper,rio2011asymptotic} as well as
\cite{bobkov2018berry} (and references therein) for convergence rates
in $p$-Wasserstein for all $p \ge 1$ under the condition of existence
of moments of order $2+p$; in all cases the rate obtained is of
optimal order $\mathcal{O}(1/\sqrt n)$.  In higher dimensions, results
are also available in 2-Wasserstein distance, under more stringent
assumptions on the $X_i$.  For instance, Zhai
\cite{zhai2016multivariate} shows that when $X_i$ is almost surely
bounded, then a near-optimal rate of convergence in
$\mathcal O(\log n / \sqrt n)$ holds (this improves a result by
Valiant et al.\ \cite{valiant2010clt}). More recently, Courtade et
al.\ \cite{courtade2017existence} attained the optimal rate of
convergence $\mathcal O(n^{-1/2})$, again in Wasserstein distance of
order 2, under the assumption that $X_i$ satisfies a Poincar\'e-type
inequality; see also \cite{fathi2018stein} for a similar result under
assumption of log-concavity. % Such assumptions, as mentioned in
% \cite{courtade2017existence}, are not directly comparable to the ones
% of Zhai \cite{zhai2016multivariate}. 
Finally we mention the work of Bonis \cite{Bonis:2015aa} where similar
estimates are investigated (in Wasserstein-2 distance) under moment
assumptions only; dependence of these estimates on the dimension is
unclear (see \cite[page 12]{courtade2017existence}).
 % , a rate of $d_2(W,Z) = \mathcal O
% (n^{-\delta/4})$ is obtained under moment assumptions on the $X_i$. 
% $\E|X_i|^{2+\delta}<\infty$ for some $\delta \in (0,2]$, then
% $$ 
% Because $d_2$ dominates $d_1$ the 1-Wasserstein distance. 

One of the key ingredients in many of the more recent above-mentioned
references is the multivariate Stein's method.  Rates of convergence
in the multivariate CLT were first obtained Stein's method by
Barbour in \cite{barbour1990stein} (see also G\"otze \cite{gotze1991rate})
whose methods (which rest on viewing the normal distribution as the
stationary distribution of an Ornstein-Uhlenbeck diffusion, and using
the generator of this diffusion as a characterizing operator) led to
the so-called \emph{generator approach to Stein's method} with
starting point equation \eqref{eq:steinmulti} and its solution given
by the classical
formula~\eqref{eq:barboursolution}. % Since then many authors have
% tackled the problem (a non-exhaustive list of references is provided
% below) of obtaining optimal rates of convergence on $D(W, Z)$ for the
% strongest possible metric $D(\cdot, \cdot)$ and under the weakest
% possible assumptions on the $X_i$.
Such an approach readily provides rates of convergence in {\it
  smooth}-$k$-Wasserstein distances, i.e.\ integral probability
metrics of the form \eqref{eq:defdistance} with $\mathcal
H$ (=
$\mathcal{H}_{(k)}$) a set of smooth functions with derivatives up to
some order $k$ bounded by $1$. Of course, the smaller the order
$k$, the stronger the distance; in particular the case
$k=1$ coincides with the classical
$1$-Wasserstein distance (and therefore also \eqref{eq:2} with
$p=1$).  In \cite[Theorem 3.1]{chatterjee2008multivariate} it is
proved that if $X_i$ has a finite moment of order
$4$ then, for any smooth $h$,
\begin{equation}
\label{eq:7}
\E[h(W)-h(Z)] \leq \frac{1}{\sqrt n} \left( \frac{1}{2}
\sqrt{\E|X_i|^4-d}\;\|\nabla h\|_\infty + \frac{\sqrt{2\pi}}{3}
\E|X_i|^3\underset{x\in\R^d}{\sup} \left\| \nabla^2 h(x)
\right\|_{op}\right),
\end{equation}
where $\|M\|_{op}$ denotes the operator norm of a matrix $M$ and
$|x|$ the Euclidean norm of a vector $x\in\R^d$. The rate $\mathcal
O(1/\sqrt
n)$ is optimal. The fourth moment conditions are not optimal, nor is
the restriction to twice differentiable test functions which implies
that \eqref{eq:7} does \emph{not} lead to rates of convergence in the
1-Wasserstein distance. Similarly, the bounds on 2-Wasserstein
distance recently obtained in \cite{courtade2017existence,
  fathi2018stein} are inspired by concepts related to Stein's method
which were introduced in \cite{ledoux2015stein}; such an approach
necessarily requires regularity assumptions on the density of the
$X_i$. Hence no simple extension of their approach can lead to rates
of convergence in Wasserstein distance with only moment conditions on
the
$X_i$ (and in particular no smoothness assumptions on the densities).
In other words, no optimal rates of convergence in Wasserstein
distance are available under moment assumptions, and they seem out of reach if based on current 
available regularity results of Stein's equation.

\subsection{Contribution}

In this paper, we study Barbour's solution \eqref{eq:barboursolution}
to the Stein equation \eqref{eq:steinmulti} and prove new regularity
results: namely, if $h$ is $\alpha$-H\"older for some
$0<\alpha\leq 1$, then for all $i,j$,
$\frac{\partial^2 f_h}{\partial x_i \partial x_j}$ is $\beta$-H\"older
for $0<\beta<\alpha$. Actually, we show the stronger estimate
\begin{equation}
\label{eq:logregularity}
\left | \frac{\partial^2 f_h}{\partial x_i \partial x_j} (x) -
  \frac{\partial^2 f_h}{\partial x_i \partial x_j} (y) \right| =
\mathcal O \left( |x-y|^\alpha \log |x-y| \right), 
\end{equation}
for $|x-y|$ small. A precise statement, with explicit constants (which
depend on $\alpha$ and on the dimension $d$), is given in Proposition
\ref{lem:ellipticregularity}.
%For $\alpha=1$, this implies that when $\|\nabla h\|_\infty < \infty$, then the derivatives of order $2$ of $f_h$ are globally $\beta$-H\"older for any $0<\beta<1$; up to now it was only known that those derivatives were bounded.
Note that from
Shauder's theory, in the multivariate case (and contrary to the
univariate one), one cannot hope in general for the second derivative
of $f_h$ to inherit the Lipschitz-regularity of $h$. Actually, Rai\v{c}
\cite{raivc2004multivariate} gives a counter-example: if
$$h(x,y) = \max\{\min\{x,y\} ,0\},$$
then $f_h$ defined by \eqref{eq:barboursolution} is twice differentiable but $\frac{\partial^2 f_h}{\partial x \partial y}$ is not Lipschitz (whereas $h$ is). We study this example in more detail in Proposition \ref{prop:raic}, which shows that (at least for $\alpha=1$), the regularity \eqref{eq:logregularity} cannot be improved in general. 

In a second step, we apply those regularity results to estimate the rate of convergence in the CLT, in Wasserstein distance.
%Our main result is the following.

\begin{thm}
\label{thm:main}
Let $(X_i)_{i\geq 1}$ be an i.i.d.\ sequence of random vectors with unit covariance matrix, and $Z \sim \mathcal N(0,I_d)$. Assume that there exists $\delta \in (0,1)$ such that $\E [|X_i|^{2+\delta}] < \infty$. Then
$$\mathcal W\left(n^{-1/2} \sum_{i=1}^n X_i, Z\right) \leq \frac{ 1}{n^{\frac{\delta}{2}}} \left[ (K_1+2\,(1-\delta)^{-1})\E|X_i|^{2+\delta} +(K_2+2d\,(1-\delta)^{-1}) \E|X_i|^{\delta} \right],$$
where $\mathcal W$ stands for the $1$-Wasserstein distance, and
\begin{align*}
K_1 &= 2^{3/2} \frac{2d+1}{ d}\frac{\Gamma(\frac{1+d}{2})}{\Gamma(d/2)}\\
K_2 &= 2\sqrt{\frac 2 \pi} \sqrt d.
\end{align*}
\end{thm}
Note that the rate in $\mathcal O (n^{-\delta/2})$ is optimal when
only assuming moments of order $2+\delta$; see
\cite{bikjalis1966estimates} or \cite{nefedova2013nonuniform}. As
mentioned previously, the bounds in Theorem~\ref{thm:main} are to our knowledge the first optimal
rates in $1$-Wasserstein distance in the multidimensional case when
assuming finite moments of order $2+\delta$ only.
%The constant behaves as ? for large $d$. 

From the previous Theorem is easily derived the following Corollary,
which gives a near-optimal rate of order $\mathcal O(\log n /\sqrt n)$
when $X_i$ has finite moment of order $3$. 
\begin{cor}
\label{cor:main}
Let $(X_i)_{i\geq 1}$ be an i.i.d.\ sequence of $d$-dimensional random vectors with unit covariance matrix. Assume that $\E [|X_i|^{3}] < \infty$. Then for $n\geq 3$ ,
$$\mathcal W\left(n^{-1/2} \sum_{i=1}^n X_i, Z\right) \leq e \, \frac{ C(d)+2(1+d)\,\log n}{\sqrt n} \; \E|X_i|^3,$$
where $C(d)=2^{3/2} \frac{2d+1}{ d}\frac{\Gamma(\frac{1+d}{2})}{\Gamma(d/2)}+2\sqrt{\frac 2 \pi} \sqrt d.$
\end{cor}
%Again the constant on the higher order term behaves, when $d$ becomes large, as $d$. 
Compared to \cite{zhai2016multivariate}, our assumption on the
distribution of $X_i$ is much weaker; however the distance used in
\cite{zhai2016multivariate} is stronger and the constants are sharper
(\cite{zhai2016multivariate} obtains a constant in
$\mathcal O(\sqrt d)$). \cite{courtade2017existence} has the advantage
of stronger rate of convergence (it is optimal when ours is
near-optimal) and stronger distance, but the drawback of a less
tractable assumption on the distribution of $X_i$ (it should satisfy a
Poincar\'e or weighted Poincar\'e inequality).

\section{Regularity of solutions of Stein's equation}

Throughout the rest of the paper, for $x,y\in \R^d$, we denote by $x\cdot y$ the Euclidean scalar product between $x$ and $y$, and $|x|$ the Euclidean norm of $x$. For a matrix $M$ of size $d\times d$, its operator norm is defined as
$$\| M \|_{op} = \underset{x\in\R^d; \,|x|=1}{\sup} |Mx|.$$
Define the $\alpha$-H\"older semi-norm, for $\alpha \in (0,1]$, by 
$$
[h]_{\alpha}=\sup_{x\neq y} \frac{|h(x)-h(x)|}{|x-y|^{\alpha}}.
$$
For a multi-index $\bold{i} = ( i_1,\ldots,i_d)\in \mathbb N^d$, the
multivariate Hermite polynomial $H_{\bold{i}}$ is defined by
$$H_{\bold{i}}(x) = (-1)^{|\bold{i}|}e^{|x|^2/2} \frac{\partial^{|\bold{i}|}}{\partial x_1^{i_1}\ldots\partial x_d^{i_d}} e^{-|x|^2/2},$$
where $|\bold{i}| = i_1+\ldots+i_d$.

Let $h\; : \; \R^d \rightarrow \R$, and $f_h$ be defined by (when the integral makes sense)
\eq
f_h(x) = -\int_0^1 \frac{1}{2t} \E \,\bar h (Z_{x,t}) \;dt \label{eq:solgen},
\qe
where
$$\bar h(x) = h(x) - \E\, h(Z),$$
and
$$Z_{x,t} = \sqrt t\, x + \sqrt{1-t} \,Z.$$

Recall that, when $h$ is smooth with compact support, then \eqref{eq:solgen} defines a solution to the Stein equation \eqref{eq:steinmulti}, see \cite{barbour1990stein,chatterjee2008multivariate}. We shall prove that this is still the case when only assuming H\"older-regularity of $h$.

\begin{prop}\label{prop:regul-solut-steins} 
Let $h\; : \; \R^d \rightarrow \R$ be a $\alpha$-H\"older function; that is, $[h]_\alpha<\infty$. Let $f_h$ be the function given by \eqref{eq:solgen}. Then:
\begin{itemize}
\item $f_h$ is twice differentiable and for $\bold{i} = (i_1,\ldots,i_d) \in \N^d$ such that $1\leq|\bold{i}|\leq 2$,
\eq
\label{eq:derivfh}
\frac{\partial^{|\bold{i}|} f_h}{\partial x_1^{i_1}\ldots\partial x_d^{i_d}} = -\int_0^1 \frac{t^{\frac{|\bold{i}|}{2}-1}}{2 (1-t)^{\frac{|\bold{i}|}{2}}} \E[ H_{\bold{i}}(Z) \bar h (Z_{x,t})] dt.
\qe
\item $f_h$ is a solution to the Stein equation \eqref{eq:steinmulti}.
\end{itemize}
\end{prop}
\begin{proof}
Fix $t\in (0,1)$. Recall $\omega(x) = (2\pi)^{-d/2} e^{-|x|^2/2}$ is the density of the standard $d$-dimensional gaussian measure. Since
$$\E\, \bar h (Z_{x,t}) = \int_{\R^d} \bar h( \sqrt t\,x + \sqrt{1-t}\, z)\; \omega(z) \;dz = (-1)^dt^{d/2}(1-t)^{-d/2}\int_{\R^d} \bar h ( u )\, \omega \left( \frac{u-\sqrt t \, x}{\sqrt{1-t}} \right)\; du,$$
we have, from Lebesgue's derivation theorem, and another change of variable,
\eq
\frac{\partial^{|\bold{i}|} }{\partial x_1^{i_1}\ldots\partial x_d^{i_d}} \E\, \bar h (Z_{x,t}) = \frac{t^{\frac{|\bold{i}|}{2}}}{(1-t)^{\frac{|\bold{i}|}{2}}} \E[H_{\bold{i}}(Z) \bar h (Z_{x,t})].
\label{eq:derivhbar}
\qe
Now note that by $\alpha$-H\"older regularity, and using the fact that $\E H_{\bold{i}}(Z_1,\ldots,Z_d) = 0$,
\begin{align}
|\E H_{\bold{i}}(Z) \bar h(Z_{x,t})| &= |\E H_{\bold{i}}(Z) (\bar h(Z_{x,t}) - \bar h(\sqrt t \,x))| \nonumber\\
& \leq \E\left[|H_{\bold{i}}(Z)| \; |Z|^{\alpha} \right] (1-t)^{\alpha/2}. \label{eq:boundhihbar}
\end{align}
Thus we can apply again Lebesgue's derivation theorem to obtain \eqref{eq:derivfh}.

Now let $\omega_{t}(x) = t^{-d/2} \omega\left( \frac{x}{\sqrt{t}} \right)$; $\omega_{1-t}$ is the density of $\sqrt{1-t} \,Z$. It is well known (and can be easily checked) that $\omega_{t}$ solves the heat equation
$$\partial_t\, \omega_t = \frac{1}{2}\, \Delta\, \omega_t.$$
We deduce (again applying Lebesgue's derivation theorem, valid since $\bar h$ has polynomial growth at infinity) that
\begin{align*}
\partial_t \E\,\bar h (Z_{x,t}) &= \partial_t \int_{\R^d} \bar h (u)\, \omega_{1-t}(u-\sqrt t \,x) \,du\\
&= -\int_{\R^d} \bar h(u) \,\partial_t \omega_{1-t} (u-\sqrt t \, x) \,du - \frac{1}{2\sqrt t}\int_{\R^d} \bar h (u) \nabla \omega_{1-t}(u-\sqrt t \, x) \cdot x \,du\\
&= -\frac{1}{2}\int_{\R^d} \bar h(u)\Delta\,\omega_{1-t} (u-\sqrt t \, x) \,du - \frac{1}{2\sqrt t}\int_{\R^d} \bar h (u) \nabla \omega_{1-t}(u-\sqrt t \, x) \cdot x\, du\\
&= -\frac{1}{2t} \Delta_x \int_{\R^d} \bar h(u)\,\omega_{1-t} (u-\sqrt t \, x) \,du + \frac{1}{2t} \nabla_x\left[\int_{\R^d} \bar h (u) \, \omega_{1-t}(u-\sqrt t \, x)\, du\right] \cdot x\\
&= -\frac{1}{2t} (\Delta - x \cdot \nabla) \E\,\bar h (Z_{x,t}).
\end{align*}
Finally,
$$\bar h(x) = \int_0^1 \partial_t \E\,\bar h (Z_{x,t}) dt =-\int_0^1 \frac{1}{2t} (\Delta - x \cdot \nabla)\E\,\bar h (Z_{x,t}) \, dt = (\Delta - x \cdot \nabla) f_h,$$
the last equality being justified by \eqref{eq:derivhbar} and the bound \eqref{eq:boundhihbar}.
\end{proof}

Before stating our main regularity results, let us give the idea behind the proof. Starting from \eqref{eq:derivfh}, we have that
$$\frac{\partial^2 f_h}{\partial x_i \partial x_j} (x) - \frac{\partial^2 f_h}{\partial x_i \partial x_j} (y) =- \int_0^1 \frac{1}{2 (1-t)} \E[(Z_iZ_j - \delta_{ij}) (\bar h (Z_{x,t})-\bar h (Z_{y,t)})] dt.$$
Using the $\alpha$-H\"older regularity of $h$, the modulus of the integrand in last integral can be bounded by
$$ \frac{1}{2 (1-t)} \E[|Z_iZ_j - \delta_{ij}|\, |\bar h (Z_{x,t})-\bar h (Z_{y,t})|] \leq C_{ij} \frac{t^{\alpha/2}}{1-t} |x-y|^\alpha,$$
$C_{ij}$ being some constant. However, the function in the right
hand-side is not integrable for $t$ close to $1$. Thus, for $\eta>0$,
we split the integral between $0$ and $1-\eta$ on the one hand (where
we can use lour previous bound), and between $1-\eta$ and $1$ on the
other hand. To bound the second integral, we remark that since
$\E[Z_iZ_j-\delta_{ij}]=0$, we have that
$$\E[(Z_iZ_j - \delta_{ij}) \bar h (Z_{x,t})] = \E[(Z_iZ_j - \delta_{ij}) (\bar h (Z_{x,t})-\bar h (\sqrt t x) )],$$
which, in modulus, is less than (again using the regularity of $h$)
$$\E[|Z_iZ_j-\delta_{ij}|\, \|Z\|^\alpha] (1-t)^{\alpha/2}.$$
The power $(1-t)^{\alpha/2}$ that is gained makes the integral
converge. Finally, we optimize in $\eta>0$.

We are concerned, however, in obtaining the best constants possible
(seen as functions of the dimension $d$ and $\alpha$); this tends to
make the proofs more technical that needed if one is only concerned
with showing regularity. For this reason, the detailed exposition of
the proof in full detail is deferred to Section \ref{sec:proof}.

We start with the regularity in terms of the operator norm of the Hessian of $f_h$.

\begin{prop}\label{lem:ellipticregularity}
Let $h\; : \; \R^d \rightarrow \R$ be a $\alpha$-H\"older function for some $\alpha \in (0,1]$. Then the solution $f_h$ \eqref{eq:barboursolution} of the Stein equation \eqref{eq:steinmulti} satisfies:
\begin{equation}\label{firstresult}
\begin{array}{clcc}
\| \nabla^2 f_{|x}-\nabla^2 f_{|y} \|_{op}&\leq \; [h]_\alpha |x-y|^\alpha \left( C_1(\alpha,d)-2\log |x-y| \right), &&\text{if }|x-y|\leq 1\\
&\leq \; C_1(\alpha,d)\;[h]_\alpha &&\text{if }|x-y| > 1, 
\end{array}
\end{equation}
where 
\eq
\label{eq:c1}
C_1(\alpha,d)=2^{\frac{\alpha}{2}+1} \frac{\alpha+2d}{\alpha
  d}\frac{\Gamma(\frac{\alpha+d}{2})}{\Gamma(d/2)}.  \qe In
particular, for all $0<\beta < \alpha $ ,
$ \frac{\partial^2 f_h}{\partial x_i \partial x_j} $ is globally
$\beta$-H\"older: \eq
\label{eq:lemma}
\| \nabla^2 f_{|x}-\nabla^2 f_{|y} \|_{op}\leq \left(C_1(\alpha,d)+\frac{2}{\alpha-\beta}\right) |x-y|^{\beta} [h]_\alpha.
\qe
It also holds the $(1+\log )$ $\alpha$-H\"older regularity
\eq
\label{eq:lemma2}
\| \nabla^2 f_{|x}-\nabla^2 f_{|y} \|_{op} \leq |x-y|^{\alpha}\left(C_1(\alpha,d)+ |\log |x-y| \,| \right) [h]_\alpha.
\qe
\end{prop}
Now we turn to the regularity of the Laplacian.

\begin{prop}
\label{prop:regu2}
Let $h\; : \; \R^d \rightarrow \R$ be a $\alpha$-H\"older function for some $\alpha \in (0,1]$. Then the solution $f_h$ \eqref{eq:barboursolution} of the Stein equation \eqref{eq:steinmulti} satisfies:
\begin{align*}
\left| \Delta f_{|x}-\Delta f_{|y} \right|\leq & \; [h]_\alpha |x-y|^\alpha \left( C_2(\alpha,d)-2\,d\,\log |x-y| \right), &&\text{if }|x-y|\leq 1\\
\leq &\; C_2(\alpha,d)\;[h]_\alpha &&\text{if }|x-y| > 1,
\end{align*}
where 
\eq
\label{eq:c2}
\begin{array}{rl}
C_2(\alpha,d)&=2^{\frac{\alpha}{2}+1} \frac{(\alpha+2d)\,\Gamma(\frac{\alpha+d}{2})}{\alpha\;\Gamma(d/2)} \text{ if } \alpha \in(0,1),\\
C_2(1,d) &= 2\sqrt{\frac 2 \pi}\sqrt d.
\end{array}
\qe
In particular, for all $0<\beta<\alpha$,
\eq
\label{eq:holderlaplacian}
|\Delta f_{|x}-\Delta f_{|y} | \leq \left(C_2(\alpha,d)+d\frac{2}{\alpha-\beta}\right) |x-y|^{\beta} [h]_\alpha.
\qe
\end{prop}

Note that Proposition \ref{lem:ellipticregularity} implies that, for $\alpha=1$, when $|x-y|$ is small, then
$$\left | \frac{\partial^2 f_h}{\partial x_i \partial x_j} (x) - \frac{\partial^2 f_h}{\partial x_i \partial x_j} (y) \right| = \mathcal O \left( |x-y| \log |x-y| \right). $$
The example given by Rai\v{c} \cite{raivc2004multivariate} shows that this rate is optimal; indeed, we have the following result.
\begin{prop}
\label{prop:raic}
Let $h \; : \; \R^2 \rightarrow \R$ be the Lipschitz function defined by $h(x,y) = \max(0,\min(x,y))$. Then
$$\frac{\partial^2 f_h}{\partial x \partial y}(u,u) - \frac{\partial^2 f_h}{\partial x \partial y}(0,0) \underset{u\rightarrow 0^+}{\sim} \frac{1}{\sqrt{2\pi}} \,u \log u .$$
\end{prop}
\noindent The proof can be found in the Appendix.

\section{Multivariate Berry-Esseen bounds in Wasserstein distance}

As anticipated, we apply the regularity results obtained in previous
section to obtain Berry-Esseen bounds in the CLT, in $1$-Wasserstein
distance.

Let $X_1,X_2,...$ be an i.i.d.\ sequence of centered, square-integrable and isotropic random vectors; that is, $\E[X_1] = 0$ and $\E[ X_1 X_1^T ] = I_d$. Let $W = n^{-1/2} \sum_{i=1}^n X_i$. We are interested in $\mathcal W_\alpha(W,Z)$ for $\alpha\in(0,1]$, where the $\alpha$-Wasserstein distance is defined as
$$\mathcal W_{\alpha}(X,Y) = \underset{\left\{ h \in C(\R^n,\R) \middle| [h]_{\alpha} \leq 1\right\}}{\sup} \E h(X) - \E h(Y).$$
As in the introduction, for $\alpha=1$, the resulting distance is
$\mathcal W := \mathcal W_1$, the classical $1$-Wasserstein distance
(that is,
$\mathcal W(X,Y) = {\sup}_{\mathcal{H}} \E h(X) - \E
h(Y)$ with $\mathcal{H}$ the collection of 1-Lipschitz functions).

We are now in a position to prove our main Theorem. We first give a more general version of it in $\alpha$-Wasserstein distances; Theorem \ref{thm:main} is just the following Theorem applied to $\alpha=1$. 
\begin{thm}
\label{thm:main2}
Let $\alpha \in (0,1]$ and $(X_i)_{i\geq 1}$ be an i.i.d.\ sequence of $d$-dimensional random vectors with unit covariance matrix. Assume that there exists $\delta \in (0,\alpha)$ such that $\E [|X_i|^{2+\delta}] < \infty$. Then
$$\mathcal W_{\alpha}\left(n^{-1/2} \sum_{i=1}^n X_i, Z\right) \leq \frac{ 1}{n^{\frac{\delta}{2}}} \left[ \left(C_1(\alpha,d)+\frac{2}{\alpha-\delta}\right)\E|X_i|^{2+\delta} +\left(C_2(\alpha,d)+d \frac{2}{\alpha-\delta}\right) \E|X_i|^{\delta} \right],$$
where $C_1(\alpha,d)$ and $C_2(\alpha,d)$ are respectively defined in \eqref{eq:c1} and \eqref{eq:c2}.
\end{thm}

\begin{remark}
From Stirling's formula, $C_1(\alpha,d)= \mathcal O ( \sqrt d)$, $C_2(\alpha,d) = \mathcal O (d^{1+\alpha/2})$ for $\alpha \in (0,1)$ and $C_2(1,d) = \mathcal O (\sqrt d)$.\end{remark}

\noindent {\bf Proof of Theorems \ref{thm:main} and \ref{thm:main2} }. Let $h$ be $\alpha$-H\"older (with $ [h]_{\alpha} \leq 1$) and $f_h$ be the solution of the Stein equation defined by Proposition \ref{prop:regul-solut-steins}. Then,
\begin{align*}
\E[h(W) - h(Z)] &= \E[ \Delta f_h(W) - W \cdot \nabla f_h(W) ]\\
&=\frac{1}{n}\sum_{i=1}^n\left[ \E[ \Delta f_h(W) -\sqrt n X_i \cdot \nabla f_h(W) \right].
\end{align*}
The following calculations already appeared in the literature (see e.g. \cite{raivc2004multivariate}), we include them here for completeness. Let $W_i = W-X_i/\sqrt n = \frac{1}{\sqrt n} \sum_{j\neq i} X_j$. By Taylor's formula, we have for some uniformly distributed in $[0,1]$ (and independent of everything else) $\theta$
$$\E[X_i\cdot \nabla f_h(W)] =\frac{1}{\sqrt n}\E \left[ X_i^T \nabla^2 f_h\left(W_i+\theta \frac{X_i}{\sqrt n}\right) X_i \right],$$
leading to
$$\E[h(W) - h(Z)] = \frac{1}{n}\sum_{i=1}^n\E\left[ \Delta f_h(W) - X_i^T \nabla^2 f_h\left(W_i+\theta \frac{X_i}{\sqrt n} \right)X_i \right].$$

Let $X_{i,j}$ be the $j$th coordinate of $X_i$. Since $W_i$ is independent of $X_i$, and $X_i$ has unit covariance matrix, we have
\begin{align*}
\E\left[X_i^T \nabla^2 f_h\left(W_i\right)X_i \right] = \sum_{j,k= 1}^d \E\left[ X_{i,j}X_{i,k} \frac{\partial^2 f_h}{\partial x_j\partial x_k}(W_i) \right]
= \sum_{j=1}^d \E\left[ \frac{\partial^2 f_h}{\partial x_j^2}(W_i) \right]
= \E[\Delta f_h(W_i)].
\end{align*}
Finally,
\[
\E[h(W) - h(Z)] = \frac{1}{n}\sum_{i=1}^n\E\left[ \Delta f_h(W) - \Delta f_h(W_i) - X_i^T \left(\nabla^2 f_h\left(W_i + \theta \frac{X_i}{\sqrt n} \right)-\nabla^2 f_h(W_i)\right)X_i \right].
\]
Note then that
$$\bigg| X_i^T \left(\nabla^2 f_h\left(W_i + \theta \frac{X_i}{\sqrt n} \right)-\nabla^2 f_h(W_i)\right)X_i\bigg| \leq |X_i|^2 \left\| \nabla^2 f_h\left(W_i + \theta \frac{X_i}{\sqrt n} \right)-\nabla^2 f_h(W_i)\right\|_{op}.$$
We now use Propositions \ref{lem:ellipticregularity} and
\ref{prop:regu2} with $\beta=\delta \leq \alpha$ to obtain:
$$
\E[h(W) - h(Z)] \leq \frac{1}{n} \sum_{i=1}^n \left[ \left(C_2(\alpha,d)+d\frac{2}{\alpha-\delta} \right)\frac{\E[|X_i|^{\delta}\theta^{\delta}]}{n^{\delta/2}} + \left(C_1(\alpha,d)+ \frac{2}{\alpha-\delta} \right)\frac{\E[|X_i|^{2+\delta}\theta^{\delta}]}{n^{\delta/2}}\right].
$$
Noting that $\E[\theta^{\beta}] \leq 1$ and rearranging, we obtain the result.
\qed

\begin{cor}
\label{cor:1}
Let $\alpha \in (0,1]$, and $(X_i)_{i\geq 1}$ be an i.i.d.\ sequence of $d$-dimensional random vectors with unit covariance matrix. Assume that $\E [|X_i|^{2+\alpha}] < \infty$. Then for $n>\exp(2/\alpha)$ ,
$$\mathcal W_{\alpha}\left(n^{-1/2} \sum_{i=1}^n X_i, Z\right) \leq e \, \frac{ C_1(\alpha,d)+C_2(\alpha,d)+2(1+d)\,\log n}{n^{\frac \alpha 2}} \; \E|X_i|^{2+\alpha},$$
where $C_1(\alpha,d)$ and $C_2(\alpha,d)$ are respectively defined in \eqref{eq:c1} and \eqref{eq:c2}.
\end{cor}

\begin{proof}
  By H\"older's and the Cauchy-Schwarz inequalities, for any
  $\delta\leq \alpha$,
  $\E|X_i|^{\delta} \leq \left( \E|X_i|^{2+\alpha}
  \right)^{\delta/(2+\alpha)}$. But by Jensen's inequality,
  $\E|X_i|^{2+\alpha} \geq (\E |X_i|^2)^{(2+\alpha)/2} =
  d^{(2+\alpha)/2} \geq 1$, so that, since $\delta/(2+\alpha)<1$,
  $\left( \E|X_i|^{2+\alpha} \right)^{\delta/(2+\alpha)} \leq
  \E|X_i|^{2+\alpha}$.  Similarly,
  $\E|X_i|^{2+\delta} \leq \left( \E|X_i|^{2+\alpha}
  \right)^{1-\frac{\alpha-\delta}{2+\alpha}} \leq \E|X_i|^{2+\alpha}$.
  Note now that the bound of Theorem \ref{thm:main2} holds for any
  $0<\delta<\alpha$. Choosing $\alpha-\delta = 2/\log n$ achieves the
  proof since $n^{-\frac{1}{\log n}} = 1/ e$.
\end{proof}

When applied to $\alpha=1$, previous corollary leads to Corollary \ref{cor:main}, which we recall here: as long as $\E|X_i|^3 < \infty$,
$$\mathcal W\left(n^{-1/2} \sum_{i=1}^n X_i, Z\right) \leq e \, \frac{ C(d)+2(1+d)\,\log n}{\sqrt n} \; \E|X_i|^3,$$
where $C(d) = 2^{3/2} \frac{(2d+1)\,\Gamma(\frac{d+1}{2})}{d\Gamma(\frac d 2)}+2\sqrt{\frac 2 \pi} \sqrt d $. \cite{zhai2016multivariate} also obtains a near-optimal rate of convergence in $\mathcal O(\log n/\sqrt n)$, but under the much stronger assumption that $|X_i| \leq \beta$ almost surely; nevertheless, the distance used in \cite{zhai2016multivariate} (the quadratic Wasserstein distance) is stronger than ours, the behaviour of the constant on the higher order term is $\mathcal O(\sqrt d)$, here we obtain $\mathcal O(d)$.

\section{Extension to higher order derivatives}

The regularity result easily extends to higher order derivatives.

\begin{prop}
Let $h \; : \; \R^d \rightarrow \R$ be a smooth, compactly supported function, and denote by  $[h]_{\alpha,p}$   a common $\alpha$-H\"older constant for all derivatives of order $p$ of $h$. Then the solution $f_h$ \eqref{eq:barboursolution} of the Stein equation \eqref{eq:steinmulti} satisfies, for all $(i_1,\ldots,i_{p+2}) \in \{1,\ldots,d\}^{p+2}$:
\begin{align*}
\bigg| \frac{\partial^{p+2}f}{\prod_{j=1}^{p+2} \partial x_{i_j}}(x) - \frac{\partial^{p+2}f}{\prod_{j=1}^{p+2} \partial x_{i_j}}(y) \bigg|\leq & \; [h]_{\alpha,p} |x-y|^\alpha \left( A-2\log |x-y| \right), &&\text{if }|x-y|\leq 1\\
\leq &\; A\;[h]_{\alpha,p} &&\text{if }|x-y| > 1,
\end{align*}
where
$$A = 2^{\alpha/2+1}\frac{\alpha+d+1}{\alpha}\frac{\Gamma(\frac{\alpha+d}{2})}{\Gamma(d/2)}.$$
In particular, all derivatives of the order $p+2$ of $f_h$ are
$\beta$-H\"older for any $0<\beta<\alpha$ and we have
$$\bigg| \frac{\partial^{p+2}f}{\prod_{j=1}^{p+2} \partial x_{i_j}}(x) - \frac{\partial^{p+2}f}{\prod_{j=1}^{p+2} \partial x_{i_j}}(y) \bigg| \leq [h]_{\alpha,p} \left( A+\frac{2}{\alpha-\beta} \right)|x-y|^{\beta}.$$
\end{prop}
\begin{proof}
  Taking derivatives in \eqref{eq:barboursolution}, we have
$$\frac{\partial^{p+2}f}{\prod_{j=1}^{p+2} \partial x_{i_j}}(x) = \int_0^1 \frac{t^{p}}{2} \E\left[\frac{\partial^{p+2}\bar h}{\prod_{j=1}^{p+2} \partial x_{i_j}} (Z_{x,t})\right] dt.$$
Next perform two Gaussian integration by parts against two indices
$i_{p+1}$ and $i_{p+2}$, say, to get
$$\frac{\partial^{p+2}f}{\prod_{j=1}^{p+2} \partial x_{i_j}}(x) = \int_0^1 \frac{t^{p}}{2(1-t)} \E\left[(Z_{i_{p+1}}Z_{i_{p+2}}- \delta_{i_{p+1}i_{p+2}})\frac{\partial^{p}\bar h}{\prod_{j=1}^{p} \partial x_{i_j}} (Z_{x,t})\right] dt.$$
Then, using the same method as in the proof of Proposition \ref{lem:ellipticregularity} (we do not give all the details here), we have
\begin{align*}
&\bigg| \frac{\partial^{p+2}f}{\prod_{j=1}^{p+2} \partial x_{i_j}}(x) - \frac{\partial^{p+2}f}{\prod_{j=1}^{p+2} \partial x_{i_j}}(y) \bigg| 
%\leq & [h]_{\alpha,p} \int_0^{1-\eta} \frac{t^{p+\alpha/2}}{2(1-t)} \E\left[|Z_{i_{p+1}}Z_{i_{p+2}}- \delta_{i_{p+1}i_{p+2}}|\right] dt\\
%& + [h]_{\alpha,p} \int_{1-\eta}^1 \frac{t^{p}}{(1-t)^{1-\alpha/2}} \E\left[|Z_{i_{p+1}}Z_{i_{p+2}}- \delta_{i_{p+1}i_{p+2}}|\,|\|Z \|^\alpha\right] dt\\
\leq - [h]_{\alpha,p} \log \eta + [h]_{\alpha,p} 2^{\alpha/2+1}\frac{\alpha+d+1}{\alpha}\frac{\Gamma(\frac{\alpha+d}{2})}{\Gamma(d/2)} \eta^{\alpha/2}.
\end{align*}
Choose $\eta = |x-y|$ if $|x-y|\leq 1$, $1$ otherwise, to get the first result, and the fact that $-\log u \leq \frac{1}{\alpha-\beta} u^{\beta-\alpha} $ if $u\leq1$ for the second one.
\end{proof}

We stress that one possible application of this Proposition would be a multivariate Berry-Esseen bound in the CLT with matching moments (i.e. assuming that the underlying random variables $X_i$ share the same first $k$ moments with the Gaussian). In this case, faster rates of convergence are expected, see \cite{gaunt2016rates}.

\section{The remaining proofs}\label{sec:proof}

{\bf Proof of Proposition \ref{lem:ellipticregularity}}. Recall that
$$ \frac{\partial^2 f_h}{\partial x_i \partial x_j} = -\int_0^1 \frac{1}{2 (1-t)} \E[(Z_iZ_j - \delta_{ij}) \bar h (Z_{x,t})] dt. $$
Since $\E[Z_iZ_j-\delta_{ij} ]= 0$, we have $\E[(Z_iZ_j-\delta_{ij})\bar h(\sqrt t x)] = 0$, so that
\begin{align}\label{rentrerlafin}
\E[(Z_iZ_j - \delta_{ij}) \bar h (Z_{x,t})] &= \E[(Z_iZ_j - \delta_{ij}) (\bar h (Z_{x,t}) - \bar h (\sqrt t x))].
\end{align}
Thus,
$$\nabla^2 f_h(x) = -\int_0^1 \frac{1}{2(1-t)} \E[ (ZZ^T - I_d) (\bar h (Z_{x,t}) - \bar h (\sqrt t x))]\,dt,$$
where $Z^T$ denotes the transpose of $Z$. Let $a = (a_1,\ldots,a_d)^T \in \R^d$ and assume that $|a| = 1$. We have

\begin{align*}
a^T \nabla^2 f_h(x) a &= - \int_0^1 \frac{1}{2(1-t)} \E[ a^T(ZZ^T - I_d)\,a \,(\bar h (Z_{x,t}) - \bar h (\sqrt t x))]dt\\
&=-\int_0^1 \frac{1}{2(1-t)} \E[ ((Z\cdot a)^2 - 1) \,(\bar h (Z_{x,t}) - \bar h (\sqrt t x))]dt.
\end{align*}
Since $|\bar h (Z_{x,t}) - \bar h (\sqrt t x) | \leq [h]_\alpha (1-t)^{\alpha/2} \|Z\|^\alpha$, we also have
\begin{align}\label{estimationimp}
|\E[ ((Z\cdot a)^2 - 1) \,(\bar h (Z_{x,t}) - \bar h (\sqrt t x))]\,|
\leq [h]_\alpha \E[ | (a\cdot Z)^2 -1| \, \|Z\|^\alpha ] (1-t)^{\alpha/2}.
%\leq &[h]_\alpha \E[ (\|Z\|^2 +1) \, \|Z\|^\alpha ] (1-t)^{\alpha/2}.
\end{align}
Let us bound $E[ | (a\cdot Z)^2 -1| \, \|Z\|^\alpha ]$. Let $(a,e_2,\ldots,e_d)$ be an orthonormal basis and $\tilde Z = (a \cdot Z, e_2 \cdot Z,\ldots, e_d \cdot Z)^T$. Then $\tilde Z \sim \mathcal N(0,I_d)$. Moreover, $E[ | (a\cdot Z)^2 -1| \, \|Z\|^\alpha ] = \E [| \tilde Z_1^2 - 1 | \| \tilde Z \|^\alpha]$. Thus,
\begin{align*}
E[ | (a\cdot Z)^2 -1| \, \|Z\|^\alpha ] &= \E [| \tilde Z_1^2 - 1 | \| \tilde Z \|^\alpha] \\
&\leq \E [( \tilde Z_1^2 +1 ) \| \tilde Z \|^\alpha]\\
&= \frac{1}{d} \sum_{i=1}^d \E [(\tilde Z_i^2 +1) \| \tilde Z \|^\alpha ]\\
&= \frac{1}{d} \E [(\| \tilde Z\|^2 +d) \| \tilde Z \|^\alpha ].
\end{align*}

For all $\beta>0$, $\E \|Z\|^\beta = \frac{2^{\frac{\beta}{2}}\Gamma(\frac{\beta+d}{2})}{\Gamma(d/2)}$. We define
\begin{equation}\label{constantetest}
C = \frac{1}{d}\E[ (\|\tilde Z\|^2 +d) \, \|\tilde Z\|^\alpha ] = 2^{\frac{\alpha}{2}} \frac{\alpha+2d}{d}\frac{\Gamma(\frac{\alpha+d}{2})}{\Gamma(d/2)}.
\end{equation}
This shows in particular that $\| \nabla^2 f_h(x)\|_{op} $ is bounded. 

Now we consider $\left| a^T \left(\nabla^2 f_h(x)-\nabla^2 f_h(y)\right)a \right|$ and split the integral into two parts. Let $\eta \in [0,1]$. We have
\begin{align}
&| a^T \left(\nabla^2 f_h(x)-\nabla^2 f_h(y)\right)a | \nonumber\\
& = \bigg| \int_0^1 \frac{1}{2 (1-t)} \E\left[ a^T(ZZ^T - I_d)\,a(\bar h (Z_{x,t})-\bar h(Z_{y,t})\right] dt\bigg|\nonumber\\
\leq & \int_0^{1-\eta} \frac{1}{2 (1-t)} \E\left[ |a^T(ZZ^T - I_d)\,a |\; |\bar h (Z_{x,t}) - \bar h(Z_{y,t})|\right] dt \nonumber\\
&+ \left| \int_{1-\eta}^1 \frac{1}{2 (1-t)} \E\left[a^T(ZZ^T - I_d)\,a (\bar h (Z_{x,t})-\bar h (Z_{y,t}))\right] \right| dt.\nonumber
\end{align}
Using the $\alpha$-H\"older regularity of $h$ for the first part of the integral and \eqref{rentrerlafin} twice in the second part together with \eqref{estimationimp} and \eqref{constantetest}, we can bound the previous quantity by 
\begin{align}
& [h]_{\alpha} |x-y|^\alpha\E\left[|(a\cdot Z)^2-1|\right]\int_0^{1-\eta} \frac{t^{\alpha/2}}{2 (1-t)} dt + [h]_{\alpha} C\int_{1-\eta}^1 (1-t)^{-1+\alpha/2} dt\label{eq:azer}\\
\leq& [h]_{\alpha} \left(- |x-y|^\alpha \log \eta + \frac{2C}{\alpha}\eta^{\alpha/2}\right),\label{eq:azer2}
\end{align}
where to obtain \eqref{eq:azer2}, we used the facts that $\E\left[|(a\cdot Z)^2-1|\right] \leq 2$ and $t^{\alpha/2}\leq1$. Choose $\eta = |x-y|^2$ if $|x-y|\leq 1$, $\eta= 1$ otherwise to get \eqref{firstresult}.
Equation \eqref{eq:lemma2} is a straightforward reformulation since $1+|\log (u) |\geq1$. To get \eqref{eq:lemma}, simply note that for $0<\beta < \alpha $ and $0<u\leq 1$, $-\log u \leq \frac{1}{\alpha-\beta} u^{\beta-\alpha}$ and for $ 1\leq u$, $1\leq u^{\beta}$. \qed\\

\noindent {\bf Proof of Proposition \ref{prop:regu2}.} The regularity
of the Laplacian is proved in a similar manner as for the operator
norm of the Hessian; we do not detail the computations here. Let
$\alpha \in (0,1)$. We have
\begin{align*}
&|\Delta f_{|x} - \Delta f_{|y }|\\
& = \bigg| \int_0^1 \frac{1}{2 (1-t)} \E\left[\sum_{i=1}^d (Z_i^2 - 1) (\bar h (Z_{x,t})-\bar h(Z_{y,t})\right] dt\bigg|\\
\leq & \int_0^{1-\eta} \frac{1}{2 (1-t)} \E\left[\bigg|\sum_{i=1}^d (Z_i^2 - 1)\bigg| \; |\bar h (Z_{x,t}) - \bar h(Z_{y,t})|\right] dt \\
&+ \int_{1-\eta}^1 \bigg| \frac{1}{2 (1-t)} \E\left[\sum_{i=1}^d(Z_i^2 - 1) (\bar h (Z_{x,t})-\bar h (Z_{y,t}))\right] \bigg| \, dt\\
\leq & [h]_{\alpha} |x-y|^\alpha\E\left[\| Z\|^2+d\right]\int_0^{1-\eta} \frac{t^{\alpha/2}}{2 (1-t)} dt + [h]_{\alpha} \E[ (\|Z\|^2+d)\|Z\|^\alpha]\int_{1-\eta}^1 (1-t)^{-1+\alpha/2} dt\\
\leq& [h]_{\alpha} \left(- d\,|x-y|^\alpha \log \eta + \frac{2\E[ (\|Z\|^2+d)\|Z\|^\alpha]}{\alpha}\eta^{\alpha/2}\right).
\end{align*}
Note that
$$\E[ (\|Z\|^2+d)\|Z\|^\alpha] = \frac{2^{\frac{\alpha}{2}+1}\Gamma(\frac{\alpha+d}{2}+1)+d\, 2^{\frac{\alpha}{2}}\Gamma(\frac{\alpha+d}{2})}{\Gamma(d/2)}= 2^{\frac{\alpha}{2}} (\alpha+2d)\frac{\Gamma(\frac{\alpha+d}{2})}{\Gamma(d/2)},$$
and choose again $\eta = |x-y|^2$ if $|x-y|\leq1$, $\eta=1$ otherwise.

We can obtain better constants in the case $\alpha = 1$. Indeed, note that by using only one integration by parts,
\begin{align*}
\E\left[\sum_{i=1}^d(Z_i^2 - 1) (\bar h (Z_{x,t})-\bar h (Z_{y,t}))\right] &= \sqrt{1-t}\E\left[\sum_{i=1}^d Z_i (\partial_i\bar h (Z_{x,t})-\partial_i\bar h (Z_{y,t}))\right] \\
&=\sqrt{1-t} \E\left[ Z \cdot (\nabla h(Z_{t,x}) - \nabla h(Z_{t,y}) \right],
\end{align*}
whose modulus can be thus bounded by
$$2\sqrt{1-t}\,\E[ \| Z \| ] = \sqrt{1-t}\frac{2 \sqrt 2}{\sqrt \pi} \sqrt d.$$
Using this bound in the integral between $1-\eta$ and $1$, and choosing $\eta$ as in Proposition \ref{lem:ellipticregularity}, we obtain the results. \qed \\

\noindent{\bf Proof of Proposition \ref{prop:raic}.}
Let $u>0$. Denote $Z_i^{t,u} = \sqrt t\, u + \sqrt{1-t} \,Z_i$. We have
\begin{align*}
\frac{\partial^2 f_h}{\partial x \partial y}(u,u) &= - \int_0^1 \frac{1}{2(1-t)} \E[Z_1 Z_2 h ( Z_1^{t,u},Z_2^{t,u})]\,dt\\
&= - \int_0^1 \frac{1}{2(1-t)} \E\left[Z_1 Z_2 (\mathbf 1_{Z_2^{t,u} \geq Z_1^{t,u}\geq 0} Z_1^{t,u} + \mathbf 1_{Z_1^{t,u} \geq Z_2^{t,u}\geq 0} Z_2^{t,u} )\right]\,dt\\
&= - \int_0^1 \frac{1}{1-t} \E\left[Z_1 Z_2 \mathbf 1_{Z_2^{t,u} \geq Z_1^{t,u}\geq 0} \, Z_1^{t,u} \right]\,dt\\
&= - \frac{1}{\sqrt{2\pi}}\int_0^1 \frac{1}{1-t} \E\left[Z_1 e^{-\frac{Z_1^2}{2}}\mathbf 1_{Z_1^{t,u}\geq 0} \, (\sqrt t \,u + \sqrt{1-t}\, Z_1)\right]\,dt,\\
\end{align*}
since $\E[Z_2 \mathbf 1_{Z_2^{t,u} \geq Z_1^{t,u}} \; | \; Z_1] = \frac{1}{\sqrt{2\pi}} \int_{Z_1}^{+\infty} z\,e^{-z^2/2} dz = \frac{1}{\sqrt{2\pi}} e^{-Z_1^2/2}.$ Now,
\begin{align*}
&\int_0^1 \frac{1}{1-t} \E\left[Z_1 e^{-\frac{Z_1^2}{2}}\mathbf 1_{Z_1^{t,u}\geq 0} \, \sqrt t \,u \right]\,dt \\
& =  u \int_0^1 \frac{\sqrt t}{1-t} \E\left[Z_1 e^{-\frac{Z_1^2}{2}}\mathbf 1_{Z_1\geq -\sqrt{\frac{t}{1-t}}u} \right]\,dt\\
& = u \int_0^1 \frac{\sqrt t}{1-t} \int_{-\sqrt{\frac{t}{1-t}}u} ^{+\infty} z e^{-z^2}\, dz\,dt\\
& = \frac{u}{2} \int_0^1 \frac{\sqrt t}{1-t} e^{-\frac{t}{1-t}u^2} \,dt\\
& = \frac{e^{u^2}u}{2} \int_0^{1/u^2} \frac{\sqrt {1-u^2\,t}}{t} e^{-\frac{1}{t}} \,dt.
\end{align*}
It is readily checked that the last integral is equivalent to
$-2\log u$, when $u \rightarrow 0^+$. On the other hand, by Fubini's
theorem, we have
\begin{align*}
& \int_0^1 \frac{1}{\sqrt{1-t}} \E\left[Z_1^2 e^{-\frac{Z_1^2}{2}}(\mathbf 1_{Z_1^{t,u}\geq 0} - \mathbf 1_{Z_1^{t,0}\geq 0})\right]\,dt\\
& = \int_0^1 \frac{1}{\sqrt{1-t}} \int_{-\sqrt{\frac{t}{1-t}}u}^0 z^2 e^{-z^2}dz\,dt\\
& = \int_{-\infty}^0 z^2 e^{-z^2}\int_{\frac{z^2}{u^2+z^2}}^1 \frac{1}{\sqrt{1-t}} dt \, dz\\
& = \frac{u}{2}\int_{-\infty}^0 \frac{z^2}{\sqrt{u^2+z^2}} e^{-z^2} \, dz,
\end{align*}
which is a $\mathcal O(u)$ as $u\rightarrow 0^+$. This achieves the proof. \qed

\section*{Acknowledgements}
We would like to thank Guillaume Carlier for suggesting the problem to us  and subsequent useful discussions. We also thank Max Fathi for useful discussions.
The three authors were supported by the Fonds de la Recherche Scientifique - FNRS under Grant MIS F.4539.16.

\bibliographystyle{plain}
\bibliography{Berryessen}

\end{document}